\documentclass[letterpaper, 10 pt, conference]{ieeeconf}  

\IEEEoverridecommandlockouts                              
\usepackage[dvipdfm]{graphicx} 
\usepackage{subfigure}
\usepackage{amsmath}
\usepackage{amsfonts}
\usepackage[psamsfonts]{amssymb}
\usepackage{color}
\usepackage{comment}
\usepackage{times} 

\newtheorem{definition}{Definition}

\newtheorem{lemma}{Lemma}
\newtheorem{corollary}{Corollary}
\newtheorem{proposition}{Proposition}

\newtheorem{remark}{Remark}

\usepackage{epsfig} 

\title{\LARGE \bf
Approximate IPA: Trading Unbiasedness for Simplicity
}

\author{Y. Wardi$^\dag$\thanks{$^\dag$School of Electrical and Computer Engineering, Georgia Institute of
Technology, Atlanta, GA 30332, USA. Email: ywardi@ece.gatech.edu. Research supported in
part by  NSF under Grant CNS-1239225.} and
C.G. Cassandras$^*$\thanks{$^*$Division of Systems Engineering
and Center for Information and Systems Engineering,
Boston University,
15 St. Mary's St., Brookline, MA 02446. Email:
cgc@bu.edu. Research supported in part by NSF under Grants EFRI-0735974 and CNS-1239021, by AFOSR under Grant FA9550-12-1-0113, by ONR under Grant N00014-09-1-1051, and by ARO under Grant W911NF-11-1-0227.}
}
\begin{document}

\maketitle
\thispagestyle{empty}
\pagestyle{empty}

\begin{abstract}
When Perturbation Analysis (PA)  yields unbiased sensitivity estimators for expected-value performance functions in discrete event dynamic systems,
it can be used for performance optimization of those functions. However, when PA is known to be unbiased, the complexity
of its  estimators often does not scale  with the system's size. The purpose of this paper is to suggest an alternative approach to  optimization which balances precision with computing efforts
 by trading off complicated, unbiased PA estimators for simple, biased approximate estimators.
 Furthermore, we provide guidelines for developing such  estimators, that are largely based on
the Stochastic Flow Modeling framework. We  suggest that if the relative error (or bias) is not too large, then
optimization algorithms such as stochastic approximation
converge to a (local) minimum just like in the case where no approximation is used.
 We apply this approach to an example of
balancing loss with buffer-cost in  a finite-buffer queue, and prove a crucial upper bound on
the relative error. This paper presents the initial
study of the proposed approach, and we believe that if the idea gains traction then it may lead to
a significant expansion of the scope of PA in optimization of discrete event systems.
\end{abstract}

\section{Introduction}

Perturbation Analysis (PA) was proposed as a sample-path sensitivity-analysis technique
for performance functions defined on the state trajectories of discrete event dynamic systems, especially queueing networks
\cite{Ho91, Cassandras99}.
Two major branches of PA have evolved: Infinitesimal Perturbation Analysis (IPA), and Finite Perturbation Analysis (FPA).
IPA is suitable for situations where the sample performance functions are differentiable and it computes their
gradients, while FPA computes finite differences, and it is tailored to situations where the controlled parameter is discrete \cite{Cassandras99a}. These types of PA estimators can be
used in sample-based optimization as long as they are statistically unbiased.

The bulk of the development of PA in the past three decades has focused on IPA. However, since its inception,
IPA  has
been limited by the fact that it is unbiased only for the simplest kinds of systems, especially in the context of queueing networks  \cite{Ho91, Cassandras99}. Consequently, the main thrust of
research in PA has focused on ways to derive unbiased IPA gradient estimators. FPA has been explored
as well, with the aim of deriving exact and unbiased finite-difference estimators for classes of networks and
performance functions. Various techniques have emerged, including reparameterization of the underlying probability space
that yields unbiased IPA \cite{Gong87}, and methods of cutting-and-pasting the state trajectories for computing exact FPA estimators \cite{Ho88}.
These, however, do not scale well with the network size and typically require prohibitive computing workloads that cast doubt on
their eventual utility in applications.

Motivated by the biasedness problem in IPA, recently we explored abstractions of the event-driven dynamics into ``flows'' (e.g., fluid queues), resulting in an alternative modeling framework called {\it Stochastic Flow Models (SFM)} \cite{Cassandras02, Wardi02,Cassandras10}. Preliminary investigations indicated that in this framework
IPA is unbiased in a far-larger class of systems than in the traditional queueing setting, and its gradient
estimators often admit very-simple algorithms. Furthermore, the following observation was made
from empirical simulation results \cite{Cassandras02, Cassandras06}: Used in conjunction with gradient-optimization
  methods, IPA gradients that are derived from an SFM  can be applied successfully to sample paths of discrete event systems.
This point, explained in detail in the sequel, supports the use of SFM-derived IPA algorithms for optimizing
  discrete-event models, although the IPA derived from the latter models are biased.

This observation raises the following question: given an optimization problem
on a discrete-queueing model, when can we trust the result of an optimization algorithm that applies SFM-based IPA
to the sample paths of the discrete system? A related question concerns the special case of optimization with respect to
discrete parameters: When can we use (successfully)
a gradient-descent algorithm with SFM-based IPA?
Answers to these two questions can have practical implications if the
IPA gradients that are obtained from the SFM can be computed via very-simple algorithms.

The purpose of this paper is to present an initial investigation of the above questions. Following a general discussion of
the underlying ideas, the paper analyzes a test-case example consisting of the
loss-volume and buffer-cost    in a finite-buffer queue, as  functions of the buffer size.
 Section II presents the problem in a formal setting and recounts some background material.
Section III  analyzes the aforementioned example, and  Section IV provides simulation results.
Finally,
Section V
concludes the paper and points out directions for future research.

\section{Motivation, Problem Setting, and Preliminary Results}
Let $L(\theta):R^n\rightarrow R$ be a random function defined on a suitable probability space
$(\Omega,{\cal F},P)$, and let $\ell(\theta):=E\big(L(\theta)\big)$ be the associated expected-value function.
In situations where the gradient term $\nabla\ell(\theta)$ is sought but cannot be computed analytically,
it can be estimated by the sample gradient $\nabla L(\theta)$ or averages of independent realizations thereof.
In the setting of Discrete Event Dynamic Systems (DEDS), and especially queueing networks,
IPA often provides simple algorithms for computing the sample
 gradient $\nabla L(\theta)$ \cite{Ho91,Cassandras99}, and  accordingly this sample gradient is called the {\it IPA gradient}.

Throughout the development of the field of Perturbation Analysis (PA) it was thought that for the sample
gradient $\nabla L(\theta)$ to be useful it had to be an unbiased statistical estimator
of $\nabla\ell(\theta)$, i.e., $E\big(\nabla L(\theta)\big)=\nabla\ell(\theta)$. Since
$\ell(\theta)=E\big(L(\theta)\big)$, unbiasedness amounts to the interchangeability
of  expectation and differentiation with respect to $\theta$, and therefore, a closely-related condition
is that the random function $L(\theta)$ be
continuous w.p.1. However, in all but simple systems defined on
queueing networks, $L(\theta)$ is not continuous and its IPA gradient is biased \cite{Ho91}.
One way to get around this problem is to use the SFM framework where, for a large class of systems, the IPA gradient is
both unbiased and admits fairly simple formulas and algorithms.\footnote{Of course SFMs provide primary models for flow networks
in many systems of interest,
but in this paper we consider them as abstractions used for the purpose of  approximating DEDS.}

For example, consider a finite-buffer queue driven by the processes of arrival and service times, where jobs arriving at a full queue
 are being discarded. Suppose that  each job has a measure of quantity that is proportional to the amount of buffer it
 occupies as well as to its service time.
 Furthermore, each job is admitted to the buffer if and only if there is sufficient space to
 store it,
 and is wholly discarded otherwise.
 Suppose that the queue evolves over  a given finite time-horizon
 $[0,t_{f}]$, and consider the volume of job-quantities being lost  during that interval as a function of the buffer size.
 Thus, denoting by $\theta$ the size of the buffer,  we consider $L(\theta)$ to be the sum of the  job-quantities  being lost,
 and refer to it as the loss volume. Observe that
 $L(\theta)$ is a step function and hence the sample derivative $L^{'}(\theta)$ has the value 0 unless
 $\theta$ is jump point of $L(\cdot)$.\footnote{We use the  notation
  $L^{'}(\theta)$ for the gradient $\nabla L(\theta)$ when $\theta\in R$, and call it the
  {\it IPA derivative}.} Assuming statistical mixing, e.g., interarrival times having a density function, the probability
 that a jump in $L(\cdot)$ occurs at a given $\theta>0$ is 0. Consequently, considering a sample path at a given $\theta$,
 the IPA derivative is $L^{'}(\theta)=0$ w.p.1.  On the other hand, the expected-value function $\ell(\theta)=E\big(L(\theta)\big)$
 often is  differentiable
 and monotone-decreasing, and  hence $\ell^{'}(\theta)<0$. Thus, the IPA derivative $L^{'}(\theta)$ is
 biased and would be useless in optimization since it always yields 0.

As an abstraction of the above system,
consider the SFM  shown in Figure 1. It consists of a fluid queue with a random server's-rate process indicated by $\{\beta(t)\}$,
 and its inflow-rate process is denoted by $\{\alpha(t)\}$. The buffer size is denoted by $c$, the fluid-volume  at the buffer (workload)
 is $x(t)$, and the spillover rate
 due to full-buffer is $\gamma(t)$. We can view  $\alpha(t)$ and $\beta(t)$ as random functions of $t$, and require them to
 be piecewise continuous on the interval $[0,t_{f}]$. These functions  drive the  other queueing processes via the
 following flow equations,
 \begin{equation}
 \dot{x}\ =\ \left\{
 \begin{array}{ll}
 0, & {\rm if}\  x(t)=0\ {\rm and}\ \alpha(t)\leq\beta(t)\\
 0, & {\rm if}\  x(t)=c\ {\rm and}\ \alpha(t)\geq\beta(t)\\
 \alpha(t)-\beta(t), & {\rm otherwise},
 \end{array}
 \right.
 \end{equation}
 and
 \begin{equation}
 \gamma(t)\ =\ \left\{
 \begin{array}{ll}
 \alpha(t)-\beta(t), & {\rm if}\ x(t)=c\\
 0, & {\rm otherwise}.
 \end{array}
 \right.
 \end{equation}
 Let $\theta:=c$ be the variable parameter, and consider the sample-performance function, denoted by $L_{c}(\theta)$,
  to be the loss volume over the
 interval $[0,t_{f}]$, namely
 \begin{equation}
 L_{c}(\theta)\ :=\ \int_{0}^{t_{f}}\gamma(\theta,t)dt,
 \end{equation}
 where the dependence of the notation $\gamma(\theta,t)$ on $\theta$ highlights that fact
 that the overflow process is a function of
 $\theta$.
 Note that $L_{c}(\theta)$  can serve  as an approximation to $L(\theta)$, defined above  as the  sum
of the job-quantities being lost in the  discrete-queue setting,
 but unlike the latter, its IPA derivative is unbiased.
Furthermore, $L_{c}^{'}(\theta)$ has an extremely simple formula,
 defined as follows. Denote by $\{x(\theta,t)\}$ the buffer-workload process defined by Equation (1), and consider a sample-realization
 of
$x(\theta,t)$ for a given $\theta$ and for all $t\in[0,t_{f}]$,   as
 shown in Figure 2. We say that an associated busy period is {\it lossy} if it incurs loss during some (any) open subset thereof. Let $N$ be the (sample-dependent) number of lossy busy periods in the interval $[0,t_{f}]$. Then (see \cite{Cassandras02, Wardi02}),
 \begin{equation}
 L_{c}^{'}(\theta)\ =\ -N.
 \end{equation}
 For example, in Figure 2, $L_{c}^{'}(\theta)=-2$. This formula indicates that to compute the IPA
 derivative, all that is needed is a simple counting process.
 The simplicity of this formula was somewhat surprising since we say nothing about the processes $\{\alpha(t)\}$ or $\{\beta(t)\}$ other
 than they be piecewise continuous w.p.1.

\begin{figure}
\centering
\epsfig{file=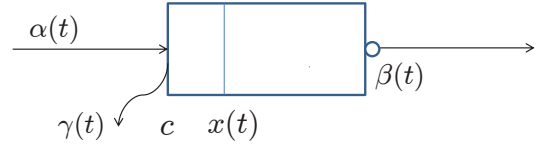,width=3.1in}\\
 \caption{Basic SFM}
\end{figure}

\begin{figure}
\centering
\epsfig{file=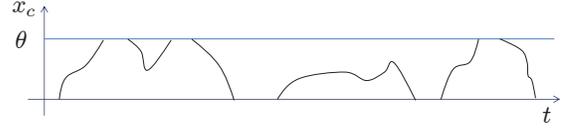,width=3.1in}\\
 \caption{Graph of $x_{c}(t)$}
\end{figure}

 The Right-hand Side (RHS) of Equation (4) can be applied to gauge
 the sensitivity of $J$ not only in the SFM context, but in the discrete-queue
  setting as well. After all, this formula is independent of whether the queue is continuous or discrete. Simulation results
   indicated that it gives
  a good approximation of $\ell^{'}(\theta)$ in the discrete case as long as it
  is approximated well by the SFM.
  Recalling  that the IPA derivative in the discrete model gives 0,  the corresponding  estimator defined by the RHS of (4)
  is clearly better. The question is, can we use it effectively in optimizing the discrete model?
  If this answer is `yes' then we have the following situation: To optimize performance of the discrete model we cannot use
  IPA derived from it, but we can  apply to its sample paths
  IPA derivatives that are based on an analysis of the SFM model. In short, the formula that is derived from the  SFM
  is applied to sample paths of the discrete model.
  The purpose of this paper is to initiate a study of when this is possible.

   Let us take the above example a step further. Consider the discrete model where  each job fits in a single
   storage cell. Suppose,  for simplicity of the argument,   that the service-time process is deterministic, and let  $s$
   denote its given constant value. Furthermore,
   let $\theta$ be the buffer size, $c$, an integer, and let $L(\theta)$ be the loss volume over the horizon
   $[0,t_{f}]$, namely, $s$ times the number of  jobs lost during the that horizon.
   Unlike the previous example $\theta$ is a discrete parameter and hence it makes no sense to
   talk about an IPA gradient. However,
  we can ask  whether the SFM IPA estimator, given by Equation (4) (multiplied by $s$),
    can serve to approximate a finite-difference
   term of the discrete model.
   Simulation results reported on in \cite{Cassandras02,Wardi02} suggest the answer can be affirmative,
   and this
   serves to motivate the present paper.

Consider now the following more-abstract setting.  Let   $L(\theta):R^n\rightarrow R$ be a random function
defined on a DEDS,
$\theta\in R^n$,
 let $\ell(\theta):=E\big(L(\theta)\big)$ denote its
expected value, and suppose that it is desirable to minimize $\ell(\theta)$.
Consider first the case where $\theta$ is a continuous variable and $L(\theta)$ is differentiable. Suppose that
$\nabla L(\theta)$ is an unbiased IPA estimator of $\nabla\ell(\theta)$, but its computation is
complicated and time consuming. The problem is how to find a random vector $h(\theta)\in R^n$ satisfying the following
two conditions: (i) $h(\theta)$ can be computed easily, and (ii) the term $||h(\theta)-\nabla L(\theta)||$ is
small enough so that stochastic-approximation algorithms that would converge (to a local minimum for $\ell$)
with the descent direction $-\nabla L(\theta)$, would
also converge with the direction $-h(\theta)$. One measure of such approximation is the relative error,
defined via
\begin{equation}
\varepsilon(\theta)\ :=\ \frac{1}{||\nabla\ell(\theta)||}||E\big(h(\theta)\big)-\nabla\ell(\theta)||,
\end{equation}
and ideally we would like to have the condition $\varepsilon(\theta)\leq\alpha$ for a given $\alpha\in(0,1)$, and for
all $\theta$. To see the reason for this, consider the deterministic case where $h(\theta)=E\big(h(\theta)\big)$. Now the inequality
$\varepsilon(\theta)\leq\alpha$ (for some $\alpha\in(0,1),\ \forall\theta$) implies  (with the aid of the triangle inequality)
that
\begin{equation}
\langle h(\theta),\nabla\ell(\theta)\rangle\geq(1-\alpha)||\nabla\ell(\theta)||^2
\end{equation}
and hence $-h(\theta)$ is a descent
direction for $\ell$, and in fact, gradient-descent algorithms with this direction would retain their essential convergence
properties \cite{Polak97}.
 In the stochastic case, asymptotic convergence of algorithms with the random
direction $-h(\theta)$ would be maintained as long as the iteration-sequence  satisfies,  asymptotically, the
limiting differential equation $\dot{\theta}=-E\big(h(\theta)\big)$. Now we mention that it may not be possible to compute
 $\varepsilon(\theta)$ since it involves expected-value quantities  which are unknown, but it can be replaced
by the sample-path relative error, defined via
\begin{equation}
{\cal E}(\theta)\ :=\ \frac{1}{||\nabla L(\theta)||}||h(\theta-\nabla L(\theta)||;
\end{equation}
in the  scalar  case where $\theta\in R$, the condition ${\cal E}(\theta)\leq\alpha$ w.p.1  implies the inequality
 $\varepsilon(\theta)\leq\alpha$ under weak assumptions, while in the vector case,
it suffices to ascertain such a sample inequality  for the partial derivatives along each one of the coordinates of $\theta$.
Finally, we mention that in some situations (but not all) the condition ${\cal E}(\theta)\leq\alpha$ may be more natural
when $||\nabla L(\theta)||$ is large rather than  small. However,   if this condition applies   whenever
$||\nabla L(\theta)||>\epsilon$ for some $\epsilon>0$, then an algorithm would converge (under suitable assumptions)
towards a set where $||\nabla\ell(\theta)||$ is bounded from above by an $\epsilon$-related quantity.
This has the practical implication of computing a parameter-point within certain
 bounds from an optimum  by an algorithm
that is not proven to compute an optimal parameter value.

Consider next the case where $\theta$ is a discrete parameter having values in a countable subset
of $R^n$. Obviously the gradient $\nabla L(\theta)$ does not exist, and an optimization algorithm could compute
$\theta_{i+1}$ from $\theta_{i}$ by examining various points $\theta$ in a neighborhood of $\theta_{i}$, and choosing the one
yielding the largest descent in $L$. To this end it is natural to use FPA for concurrent estimation
 \cite{Cassandras99a},
but this typically is a complicated and time-consuming procedure.  However, there often exists
a natural relaxation of the underlying
DEDS by an SFM that yields an approximation of $L(\theta)$ via a continuous-parameter function, $L_{c}(\theta)$;
see  \cite{Yao11,Yao12}.  Furthermore, as in the example mentioned earlier, the IPA gradient
$\nabla L_{c}(\theta)$ often is  computable via  a simple algorithm, and the idea is to
optimize $\ell(\theta)$ by using  a stochastic approximation algorithm of the form $\theta_{i+1}=\theta_{i}-\lambda_{i}\nabla L_{c}(\theta_{i})$,
$\lambda_{i}>0$. In other words, the descent direction of $L$ is determined by the IPA derivative of the SFM.
This approach is justified as long as the difference term
$\Delta L(\theta):=L(\theta+\Delta\theta)-L(\theta)$ is approximated well by the term
$\langle\nabla L_{c}(\theta),\Delta\theta\rangle$ for a small-enough $\Delta\theta$.

This is the case that is exemplified and analyzed in the rest of the paper. We mention that we are not concerned here with theoretical
issues related to asymptotic convergence of stochastic approximation, including the question of when the computed iteration-sequence
satisfies the limiting ODE. We focus only on an example where we establish the proximity of  the sample-terms
$\langle L_{c}^{'}(\theta),\Delta\theta\rangle$ and $\Delta L(\theta):=L(\theta+\Delta\theta)-L(\theta)$, which supports
the convergence of a stochastic approximation algorithm that runs on a discrete parameter space of a DEDS while using the IPA formula
derived from its SFM abstraction. This is but a first study of the general approach described above, and we will argue that
its results merit further investigations.

\section{case Study: $G/D/1/k$ Queue}
This section concerns sensitivity estimation of the loss volume
   in a $G/D/1/k$ queue with respect to variations in the buffer size, $k$. This  performance function is related to the
loss probability; see \cite{Cassandras02}.
Suppose that the service order of jobs is according to their arrival order,  denote by $s$ the (constant) service time
of each job, and suppose that each job requires one buffer unit for its storage.  We point out that the assumption
of deterministic service times can be relaxed in two ways: one assumes a continuous buffer size in which every job
is allocated an amount of buffer that is proportional to its service time, and the other assumes that each job is stored in a single
buffer unit regardless of its service time. In both cases the essential elements of the analysis in the sequel remain
unchanged but the technicalities increase in complexity, and for this reason we make the assumption of a  constant service time,
which suffices to capture the gist of the method that we propose.

Consider the sample-path evolution of the queue during a given time-interval $[0,t_{f}]$ as a function of the buffer size $k$. This
means that a given $k$ is fixed throughout the above time-interval. Suppose, for the sake of simplicity, that
the queue is empty at time $t=0$. Let $n(k)$ denote the number of jobs lost due to
 a full  buffer in the above time-interval,
and let $x(k,t)$ denote the number of jobs at the queue (server plus buffer) at time $t$. We will consider the loss-volume performance function throughout the horizon $[0,t_{f}]$, defined by
$L(k):=n(k)\times s$.

Consider now an SFM analogue of the queue, in which the buffer capacity is a continuous variable, denoted by $\theta$,  the arrival-rate
and service-rate processes are $\{\alpha(t)\}$ and $\{\beta(t)\}$, respectively, and the workload $x_{c}(\theta,t)$
and spillover rate $\gamma_{c}(x,t)$  are defined
by Equations (1) and (2). The SFM-based performance function that is  analogous to $L(k)$ will be denoted by $L_{c}(\theta)$,
and to align it with the discrete model, where each job carries a workload of $s$, we define it as
$L_{c}(\theta)=s\int_{0}^{t_{f}}\gamma_{c}(\theta,t)dt$.
Notice that for $\theta=k$, it is not generally true that  $L_{c}(\theta)=L(k)$, since the discrete model accepts or rejects whole
 jobs  only, while the continuous-flow model permits the storage of any fluid volume. Therefore,  $s^{-1}L(k)$ has only
integer values while $s^{-1}L_{c}(\theta)$, $\theta=k$, generally has non-integer values.
 Thus, the approximation of the SFM to the discrete setting is in the model and not merely in
the performance function.

To speak of such approximations suggests that we have to define the SFM model in detail.
However, this is not quite the case. The reason is that what we need is not the function-value
$L_{c}(\theta)$ but only its IPA  derivative
$L_{c}^{'}(\theta)$, which is computable by a formula that does not depend on  the detailed model but
relies   on
a sample path of the discrete system. To clarify this point, observe that the IPA formula in (4) is
 {\it independent} of the particular forms of $\alpha(t)$ or
$\beta(t)$ though its implementation depends on the sample path, which can be obtained from the discrete system
as well as from the continuous one. Thus, when using the term $L_{c}^{'}(\theta)$ we mean the IPA derivative, derived
from the SFM model but applied to the sample path of the discrete model at a point $\theta=k$, an integer.

Consider the sensitivity measure $\Delta L(k):=L(k+1)-L(k)$, defined over the common sample path underscoring
the realizations of both $L(k+1)$ and $L(k)$.
The purpose of the IPA derivative $L_{c}^{'}(\theta)$, $\theta=k$,
 is to approximate this  finite-difference term
via the first-order approximation term
$L_{c}^{'}(\theta)\Delta\theta$, and since $\Delta\theta=1$, this   is equal to $L_{c}^{'}(\theta)$.
Specifically, we will use the notation $L_{c}^{'}(k)$ to indicate the IPA derivative $L_{c}^{'}(\theta)$
at the point $\theta=k$, computed from a
sample path that is obtained from the  discrete-event  model.
For the purpose of optimization, this can be used in a variant of a stochastic approximation algorithm
with roundoffs, as described in Section IV. What concerns us in the rest of this section is the error
term
$E(k):=|\Delta L(k)-L_{c}^{'}(k)|$, and we derive for it an upper bound in the sequel.

Fix $k\geq 1$, and consider the process $x(k,t)$, $t\in[0,t_{f}]$, corresponding to a sample path of the
discrete queue. Following standard terminology, we call  this process a {\it nominal trajectory}
 and the process $x(k+1,t)$ corresponding to the same sample path, the
 {\it perturbed trajectory} (see \cite{Ho91}).   PA yields the difference-process  $\Delta x(k,t):=x(k+1,t)-x(k,t)$ from the nominal trajectory, and in order to describe its algorithm we  define (below) two types of events: one is the occurrence
 of conditions where $\Delta x(k,t)$ increases from 0 to 1, and the other is the start of a busy
 period according to the nominal trajectory. The significance of these events will be made clear following their
 formal definition, which requires  two auxiliary variables:
 a binary variable $\psi \in \{0,1\}$, and
a real variable $\zeta \in [0,s]$. These are defined together with the event types in the following recursive manner.

Given a buffer size $k>0$ and
 a sample path of the system, set
$\psi:=0$ and $\zeta:=s$ at time $t=0$.
\begin{definition}
\begin{enumerate}
\item
A type-1 event is the arrival time of a job, $t_a$, when
(i) $x(k,t_a^-) = k$; (ii) $\psi = 0$; and (iii)  with $t_{0}$ denoting the service-starting time of the job currently in the server at time
$t_{a}$, $\zeta > s - (t_a-t_0)\geq0$.
 When a type-1 event occurs at time $t_{a}$, set $\psi=1$ and set $\zeta=0$.
\item
A type-2 event is the arrival time of a job, $t_{b}$, when $x(k,t_{b}^-)=0$.   When such an event occurs at time $t_{b}$, set
 $\psi=0$,
 and, with  $\tau_{e}$ denoting the time the previous busy period ended, set
\begin{equation}
\zeta\ =\ \min\{t_{b}-\tau_{e}+\zeta,s\}.
\end{equation}
\end{enumerate}
\end{definition}
\hfill$\Box$

Let us denote by $\zeta(t)$ and $\psi(t)$ the values of the variables $\zeta$ and $\psi$ at time $t$, as computed by the processes described in Definition 1.
\begin{remark}
Observe that type-1 events occur when a job is turned away from the queue due to a full buffer, according to the nominal trajectory,
subject to the three conditions specified in Definition 1.1; later we will show that these conditions guarantee that type-1 events occur when
$\Delta k(t_{a})$ is increased from 0 to 1. A type-2 event is the start of a busy period according to the nominal trajectory.

The term $\zeta$ represents the delay of the service-schedule of the nominal trajectory with respect to the perturbed trajectory during
periods when $\Delta x(k,t)=1$. To see this, note that initially $\Delta x(k,t)=0$
until the first time a job is turned away according to the nominal
trajectory. This job is absorbed  by the extra buffer
according to the perturbed trajectory; henceforth $\Delta x(k,t)=1$ while the service schedule of the two trajectories are aligned, until
the queue becomes empty according to the nominal trajectory. To formalize this,  denote by $t_{a,1}$ the time the first job is discarded according to the nominal
trajectory,  and let $\tau_{e,1}>t_{a}$ denote  the end-time of the first busy period according to the nominal
trajectory.
Note that $t_{a,1}$ is the time of a type-1 event, and hence,   by Definition 1, $\zeta(t)=0$ $\forall t\in[t_{a,1},\tau_{e,1})$; furthermore,
as just explained, $\Delta x(k,t)=1$
$\forall t\in[t_{a,1},\tau_{e,1})$.

Next, let $t_{b,1}>\tau_{e,1}$ be the time of the next job-arrival after $\tau_{e,1}$, namely the end-time
of the idle period begun at time $\tau_{e,1}$ for the nominal trajectory.
The term $t_{b,1}-\tau_{e,1}$ is the length of that idle
period. If $t_{b,1}-\tau_{e,1}\geq s$ then the extra job of the perturbed trajectory is being served during this idle period, and  at time
$t_{b,1}$ a new busy period starts for both trajectories. On the other hand, if $t_{b,1}-\tau_{e,1}<s$ then the extra job of the perturbed
trajectory
starts its service at time $\tau_{e,1}$ but does not complete it by the time the next job arrives at time $t_{b,1}$;
the period
$[\tau_{e,1},t_{b,1})$ is idle only for the nominal trajectory but not for the perturbed
one. In this case, Equation (8) sets, at time $t_{b,1}$, $\zeta(t_{b,1})=t_{b,1}-\tau_{e,1}$
(since $\zeta(t_{b,1}^{-})=0)$.  Henceforth, until either another type-1 event occurs or the buffer becomes
empty according to the nominal trajectory, the value of $\zeta$ remains unchanged, and  in every
service period (according to the nominal trajectory), the perturbed trajectory completes the service of
that job during the first $s-\zeta$ seconds, and  it
starts serving the next job during the remaining $\zeta$ seconds, while
the nominal trajectory starts the service of the same next job $\zeta$ seconds later. In other words,
if we denote by $[t,t+s)$ the service time of a job according to the nominal trajectory, then during the time-period
$[t,t+s-\zeta)$ the perturbed trajectory completes the service of this  job and it starts serving the next job at time
$t+s-\zeta$; the nominal trajectory will start serving its next job at time $t+s$. All of this will change with the next event according to its type.
If the next event is type 1 then $\zeta$ is set to 0 and the extra job of the perturbed trajectory is absorbed by the extra buffer,
 while if the next event is of type 2, then $\zeta$ is recomputed by (8).
\end{remark}

The implications of all of this on the process $\{\Delta x(k,t)\}$
 are summarized in the following assertion.

\begin{lemma}
\begin{enumerate}
\item
Let $t_{a}$ be the time of a type-1 event, and let $\tau_{e}>t_{a}$ be the end-time of the busy period (according to the nominal
trajectory) containing $t_{a}$.
Then for every $t\in[t_{a},\tau_{e})$, $\Delta x(k,t)=1$.
\item
Let $t_{b}$ be the time of a type-2 event, and let $t_{a}>t_{b}$ be the time of the next type-1 event.
Let $[\tau,\tau+s)\subset [t_{b},t_{a})$ be a service period according to the nominal trajectory.
Then,  $\Delta x(k,t)=1\  \forall t\in[\tau,\tau+s-\zeta(t))$, and $\Delta x(k,t)=0\ \forall t\in[\tau+s-\zeta(t),\tau+s)$.
\end{enumerate}
\end{lemma}
\begin{proof}
The main argument is by induction. Let $t_{a,i}$, $i=1,2,\ldots$,
denote the $ith$ time
a type-1 event occurs. Furthermore, let $\tau_{e,i}>t_{a,i}$ denote the end-time of the busy period
containing $t_{a,1}$ according to the nominal trajectory. By the assignments of the values
of the binary variable $\psi$ in Definition 1, it follows that
$t_{a,i+1}>\tau_{e,i}$; in other words, the next type-1 event following
$t_{a,i}$ can occur only after the buffer becomes empty. Let $t_{b}$ be any
type-2 event-time between $\tau_{e,i}$ and $t_{a,i+1}$, and note that there may be
several such events, since there may be multiple empty periods between the times
$\tau_{e,i}$ and $t_{a,i+1}$.

For $i=1$, the detailed remarks prior to the statement of this lemma showed that
$\Delta x(k,t_{a,1}^{+})=1$. Next, suppose that, for some $i\geq 1$,
 $\Delta x(k,t_{a,i}^{+})=1$.
We will prove the assertions of the lemma
for $[t_{a,i},\tau_{e,i})$ (part 1) and for $[t_{b},t_{a,i+1})$ (part 2), as well as that
$\Delta x(k,t_{a,i+1}^{+})=1$; this will complete the proof.

By the induction's hypothesis, $\Delta x(k,t_{a,i}^{+})=1$, and therefore, it
is clear that, for all $t>t_{a,i}$ until the buffer becomes empty or full
(according to the nominal trajectory), $\Delta x(k,t)=1$. Suppose that the buffer becomes
full before it becomes empty, namely at some  time $\bar{t}\in(t_{a,i},\tau_{e,i})$.
Then, at that time, the buffer is becoming full according to both nominal and perturbed
trajectories,  the full-buffer period starting at that time is common to
both trajectories, and the relation
 $\Delta x(k,t)=1$  remains  throughout it. Thus, this relation
 would be maintained until
 the next time the buffer becomes empty according to the nominal trajectory, namely
 time $\tau_{e,i}$. This proves part 1 of the lemma.

 Next, consider what happens at time $\tau_{e,i}$, when the buffer
 becomes empty according to the nominal trajectory. At this time $\zeta$ is computed by
 Equation (8), and as discussed earlier, for every service-period of the nominal
 trajectory, $[t,t+s)$, that occurs before the next time the buffer becomes
  empty or a loss takes place, $\Delta x(k,\tau)=1$ $\forall\tau[t,t+s-\zeta(t))$,
 and $\Delta x(k,\tau)=0$ $\forall\tau\in[t+s-\zeta,t+s)$. This situation
 will not change if  the buffer becomes empty before time $t_{a,i+1}$, except that
 $\zeta$ would be recomputed via  Equation (8) (and possibly reduced).
 Next, if a loss takes place at a time $\tilde{t}$ prior to $t_{a,i+1}$ then,
 by Definition 1.1, this occurs during the first $s-\zeta$ seconds of the current service period; $\Delta x(k,\tilde{t})=1$ during that service period, and hence the loss occurs for both the nominal and perturbed trajectories, and the above relations
 concerning $\Delta x(k,\cdot)$  remain unchanged. On the other hand, at time
 $t_{a,i+1}$, the loss associated with the nominal trajectory occurs during the
 the last $\zeta$ second of the current service period; during that time
 $\Delta x(k,\cdot)=0$, and hence $\Delta x(k,t_{a,i+1}^{+})=1$. This completes the proof.
\end{proof}

We next consider the IPA approximation of the finite-difference term $\Delta L(k)$.
Recall that the loss volume is $L(k):=n(k)\times s$ where $n(k)$ is the number of jobs lost during the interval $[0,t_{f}]$,
while in the SFM setting $L_{c}(\theta)=s\int_{0}^{t_{f}}\gamma_{c}(\theta,t)dt$. By
Equation (4),
the IPA derivative of the analogous SFM is $L_{c}^{'}(k)=-Ns$ where $N$ is the number of lossy busy periods obtained from the nominal trajectory of the discrete queue.
Recall that   $E(k):=|\Delta L(k)-L_{c}^{'}(k)|$. We are interested in an upper bound on
$E(k)$ since it will yield an upper bound on the relative error.

Consider a particular sample path, and define
 the two quantities, $N_{s}$ and $N_{\ell}$, as follows.
$N_{s}$ is  the number of lossy busy periods in the horizon $[0,t_{f}]$ whose preceding idle periods are shorter than
$s$ seconds, and   $N_{\ell}$ is  the number of lossy busy periods whose preceding idle periods are at least as long as $s$ seconds.
Naturally   $N_{s}+N_{\ell}=N$.

The main result of this subsection is the following.

\begin{proposition}
The following inequality is in force,
\begin{equation}
E(k)\ \leq\ sN_{s}.
\end{equation}
\end{proposition}
\begin{proof}
Let $N_{1}$ denote  the number of times a type-1 event occurs during the horizon $[0,t_{f}]$.
We first show that
\begin{equation}
\Delta L(k)\ =\ -sN_{1}.
\end{equation}
Let $\tilde{N}$ be the number of times a loss occurs at a time $t$ such that
$\Delta x(k,t^{-})=0$.
Observe that $-\Delta L(k)=s\tilde{N}$. Now we argue that
$N_{1}=\tilde{N}$. The proof of this is by induction. Let $\bar{t}_{1},\bar{t}_{2},\ldots$, be the successive occurrence times
of type-1 events, and let $\tilde{t}_{1},\tilde{t}_{2},\ldots,$ be  the successive times
when a loss occurs while
$\Delta x(k,\tilde{t}_{i})=0$. Suppose that, for a given $i\geq 1$,
$\bar{t}_{i}=\tilde{t}_{i}$; we next show that
$\bar{t}_{i+1}=\tilde{t}_{i+1}$.

By  Lemma 1.1, $\Delta x(k,t)=1$ for every $t$ from
$\tilde{t}_{i}$ to the end of the busy period containing $\tilde{t}_{i}$; hence
$\tilde{t}_{i+1}$ lies in a subsequent busy period.
By Definition 1, $\psi(t)=1$ from every $t$ from $\bar{t}_{i}$ to the end of the busy
period containing $\bar{t}_{i}$; hence $\bar{t}_{i+1}$ lies in a subsequent busy period.
Since by assumption $\bar{t}_{i}=\tilde{t}_{i}$, both $\bar{t}_{i+1}$ and $\tilde{t}_{i+1}$ lie in a  busy period
subsequent to the one containing $\bar{t}_{i}$.
 Next, Lemma 1.2, with the aid of Definition 1, implies that in any subsequent busy period, a type-1 event occurs at and only at
 a time $t$ when a loss occurs while $\Delta x(k,t^{-})=0$. This establishes that $\bar{t}_{i+1}=\tilde{t}_{i+1}$, and hence that $N_{1}=\tilde{N}$. Since $-\Delta L(k)=s\tilde{N}$, Equation (10) follows.

 Observe that, by Definition 1, $\psi=1$ between any type-1 event and the following type-2 event,
 and $\psi=0$ between any type-2 event and the following type-1 event.
 Furthermore, the condition $\psi=0$ is required for a type-1 event to occur. Therefore, there can be at most a single
 type-1 event
 in a lossy busy period (according to the nominal trajectory), implying that  $N_{1}\leq N$.

 Next, consider a type-2 event occurring at a time $\tau$, and suppose that the preceding idle period was no shorter
 than $s$ seconds. By (8), $\zeta(\tau)=s$
 (since, in the notation of (8), $t-\tau_{0}\geq s$). Moreover,
 for every subsequent type-2 event occurring after $\tau$ and prior to the next loss,
 $\zeta$ will retain its value of $s$ (see (8) with the condition $\zeta\geq s$, implying that $t-\tau_{0}+\zeta\geq s$).
 Therefore, for every $t$ between $\tau$ and the next loss, $\zeta(t)=s$. Consequently, and by Definition 1.1, the next loss must be a type-1 event. We conclude that the first loss following a type-2 event whose preceding idle period is no shorter
 than $s$ second, must be a type-1 event, and this implies that $N_{\ell}\leq\ N_{1}$.

 In summary, we have that $N_{\ell}\leq N_{1}\leq N=N_{s}+N_{\ell}$. By Equation (10), $\Delta L(k)=-sN_{1}$; by Equation (4),
 $L_{c}^{'}(k)=-sN$; and consequently,
 \begin{equation}
 E(k):=|\Delta L(k)-L_{c}^{'}(k)|=s(N-N_{1})\leq sN_{s}.
 \end{equation}
This completes the proof.

\end{proof}

Similarly to (7), let  us define the relative error by
\begin{equation}
{\cal E}(k)\ :=\ \frac{E(k)}{|L_{c}^{'}(k)|}.
\end{equation}
\begin{corollary}
The following inequality holds,
\begin{equation}
{\cal E}(k)\ \leq\ \frac{N_{s}}{N}\ \leq\ 1.
\end{equation}
\end{corollary}
\begin{proof}
Immediate by Proposition 1 and the fact that $L_{c}^{'}(k)=-sN$.
\end{proof}

We will use the SFM-IPA derivative $L_{c}^{'}(k)$ in a stochastic approximation algorithm applied to
the discrete system, where its parameter is the buffer size, $k$. Obviously $N_{s}/N\ \leq\ 1$, and moreover,
under broad assumptions (such as the inter-arrival times being iid),
the expected-value of $N_{s}/N$ is bounded from above by the probability that an inter-arrival time is less than
 $s$, which is less than 1 under stability conditions. Thus, Corollary 1 indicates that  $-L_{c}^{'}(k)$ is a descent direction for $L$ as long as
$L_{c}$ provides a good approximation to $L$. This   would be  the case   when the service times and inter-arrival times
are very short while their ratio is less than 1.  An example in the next section will demonstrate this point.

\section{Simulation experiments}

Consider the problem of balancing the loss volume with a buffer-cost, cast in the form of minimizing a weighted sum
of these two performance functions. Specifically, with a given $a>0$ representing the cost per unit buffer,
we consider the sample-performance function $F(k):=L(k)+ak$, and attempt to minimize its
expected-value, $f(k):=\ell(k)+ak$. To this end we employ a stochastic approximation algorithm having the following form.

Given a sequence of step sizes $\lambda_{i}>0$, $i=1,2,\ldots$, satisfying $\sum_{i=1}^{\infty}\lambda_{i}=\infty$ and
$\sum_{i=1}^{\infty}\lambda_{i}^2<\infty$.
The algorithm enters iteration $i$ with $\theta_{i}\in R$. First, it sets $k_{i}$ to be the closest integer
to $\theta_{i}$. Then it runs and observes a sample path of the queue at $k_{i}$, based on which it
computes the IPA derivative $L_{c}^{'}(k_{i})$, and the sample derivative
$F_{c}^{'}(k_{i}):=L_{c}^{'}(k_{i})+a$. It then considers the displacement term from $k_{i}$, namely the product term
$\lambda_{i}\times F_{c}^{'}(k_{i})$. It may have to scale this term if it is too large, and hence it defines, for a given
$r>0$,
\begin{equation}
d_{i}\ :=\ \left\{
\begin{array}{ll}
\lambda_{i}F_{c}^{'}(k_{i}), & {\rm if}\ |\lambda_{i}F_{c}^{'}(k_{i})|\leq r\\
r\times{\rm sign}F_{c}^{'}(k_{i}), & {\rm if}\ |\lambda_{i}F_{c}^{'}(k_{i})|>r.
\end{array}
\right.
\end{equation}
Finally, the algorithm sets $\theta_{i+1}:=\theta_{i}-d_{i}$.

Note that the principal output of the algorithm is the sequence of
integers $\{k_{i}\}$, but it computes them by iterating on the sequence $\{\theta_{i}\}$ of
real numbers. The reason is that whereas we compute the IPA at the integer-values of the buffer size, letting the algorithm
compute its iterations only among such values may cause it to jam due to the fact that the step sizes decline to zero
at an a-priori rate. The use of the auxiliary variable $\theta_{i}$ basically prevents this by accumulating
fractional descents during several iterations. We also mention that the truncation in Equation (14) is made in order to prevent
large values of the IPA derivative, due to statistical fluctuations, from destabilizing the algorithm.

The system considered is an M/D/1/k queue with the arrival rate of 90 jobs-per-second and service time of $s=0.01$ seconds,
evolving over the horizon interval $t\in[0,20]$.
The algorithm was run with the following parameters:
the buffer-unit cost is $a=0.2$, the step size is $\lambda_{i}=10/i^{0.6}$, and the threshold
parameter is $r=2.5.$.

Results of a typical run for 100 iterations, with the initial value of $k=15$,  are shown in Figures 3-5.
Figure 3 depicts the graphs of the buffer size $k$ (solid curve) and the variable $\theta$  (dotted curve)
as  functions of the iteration count, $i$. We discern  a descent of $k$ from its initial value of 15 towards the values of 6 and 7; this was corroborated by
several simulation results (not shown here) verifying that $F_{c}^{'}(k):=L_{c}^{'}(k)+a$ is positive for $k\geq 7$ and negative for
$k\in\{1,\ldots,6\}$. The final value of $\theta$ is 6.49. Figures 4 and 5 show the graphs of $F(k)$ and $F_{c}^{'}(k)$, respectively, both as functions
of the iteration count, and exhibit a decline in the sample cost and an approach of its IPA derivative to 0, up to oscillations that are due to statistical fluctuations; neither graph is surprising. The second experiment starts at the initial buffer size
of $k=1$, and the convergence, more dramatic than for the first experiment, is indicated in Figures 6-8, where $k$ approaches
the 6-7 range and the final value of $\theta$ is 6.47.

\section{Conclusions}

In this paper we propose an approach to  the use of IPA in optimization of discrete-event and hybrid dynamical
systems,  which is based on a tradeoff between precision and computational efforts. Our interest in this issue is
motivated by the fact that unbiased sensitivity gradient estimation, such as FPA or IPA, may require prohibitive simulation
and computing costs, but an adequate approximation via an IPA gradient of a related (but different) model often
yields a desirable descent direction while involving considerably less computations. In particular, this can be the case in situations where the performance metric
of interest is defined on a discrete-event model, and the approximating sensitivity analysis estimate
is the IPA of a related performance metric defined on an SFM.

Following a presentation of this idea, the paper tested it on a particular example of the loss volume in a finite-buffer queue as a function
of the buffer size, which is a discrete parameter. The sensitivity analysis estimator with respect to this parameter, FPA,
is approximated by the IPA of an analogous function defined on a related SFM.
Our analysis includes the derivation of an upper bound on the relative error that is expected to yield convergence
of a stochastic approximation algorithm using the approximate IPA, and this is supported by simulation experiments.

The analysis in the paper focuses on the derivation of upper bounds on the relative errors between the
sensitivity estimators of the discrete model and those of the related SFM. This analysis was exact but rather tedious for the system
considered in this paper,  and its arguments may not be extendable to a general setting of queueing networks or DEDS.
To get around this problem we plan on pursuing an alternative approach to error analysis that is based on a-posteriori bounds  derived (easily) from the sample paths of the system, which yield sufficient information to ascertain the effectiveness of the
approximate model on the behavior of an optimization algorithm.


\begin{figure}[ht!]
\centering
\includegraphics[scale=0.50]{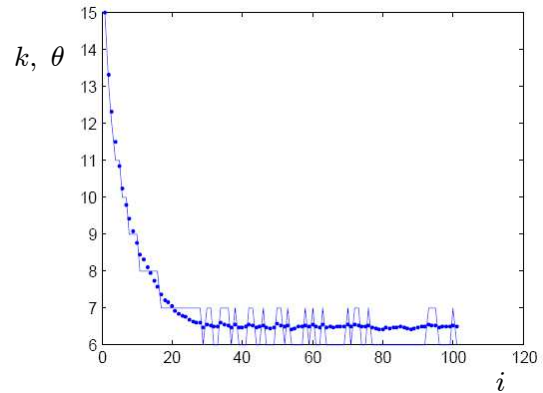}
\caption{$k$ and $\theta$ vs. iteration count ($i$), experiment 1.}
\label{fig:fig3}
\end{figure}

\begin{figure}[ht!]
\centering
\includegraphics[scale=0.48]{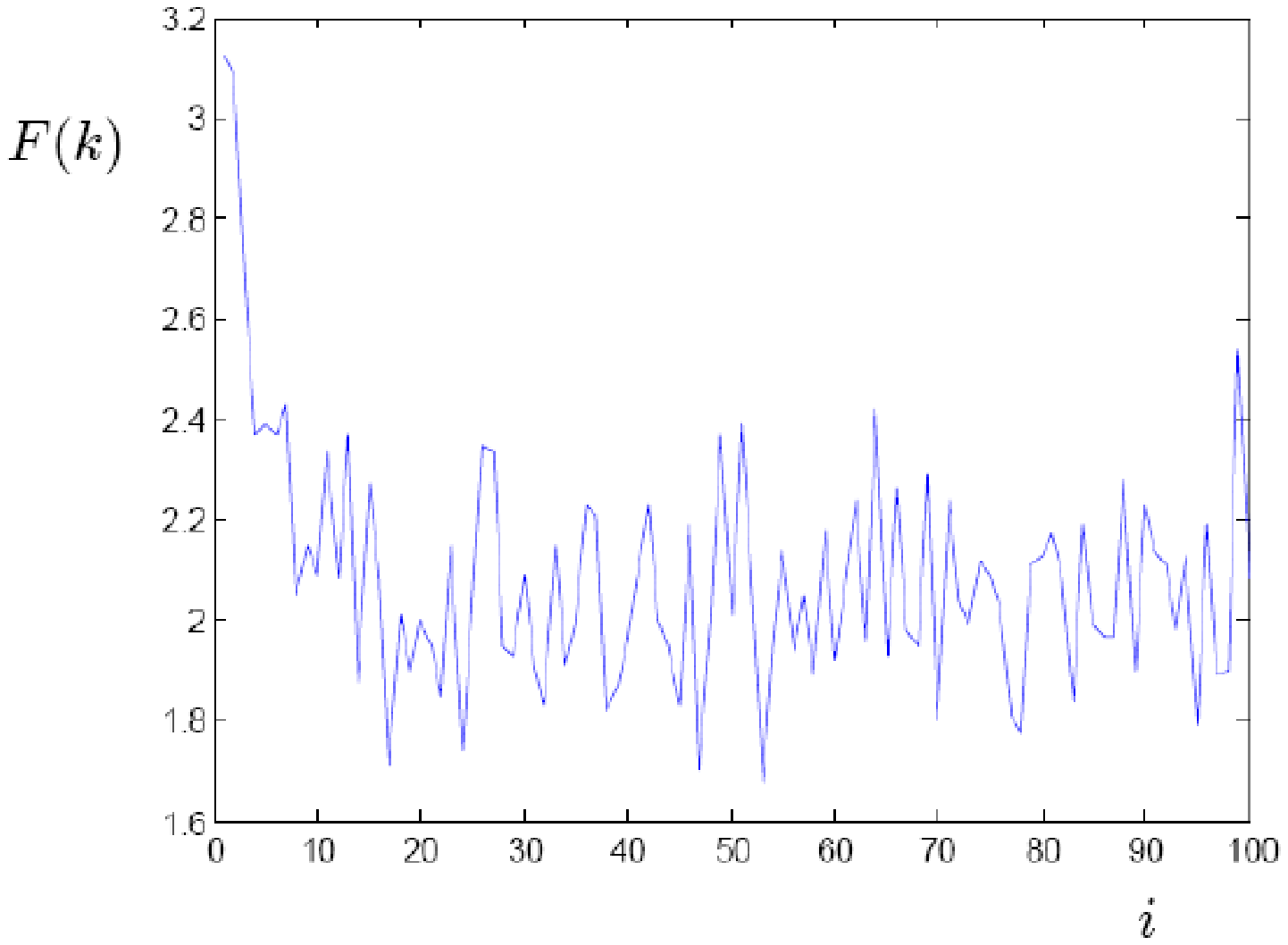}
\caption{ $F(k)$ vs.  iteration count, experiment 1.}
\label{fig:fig3}
\end{figure}

\begin{figure}[ht!]
\centering
\includegraphics[scale=0.40]{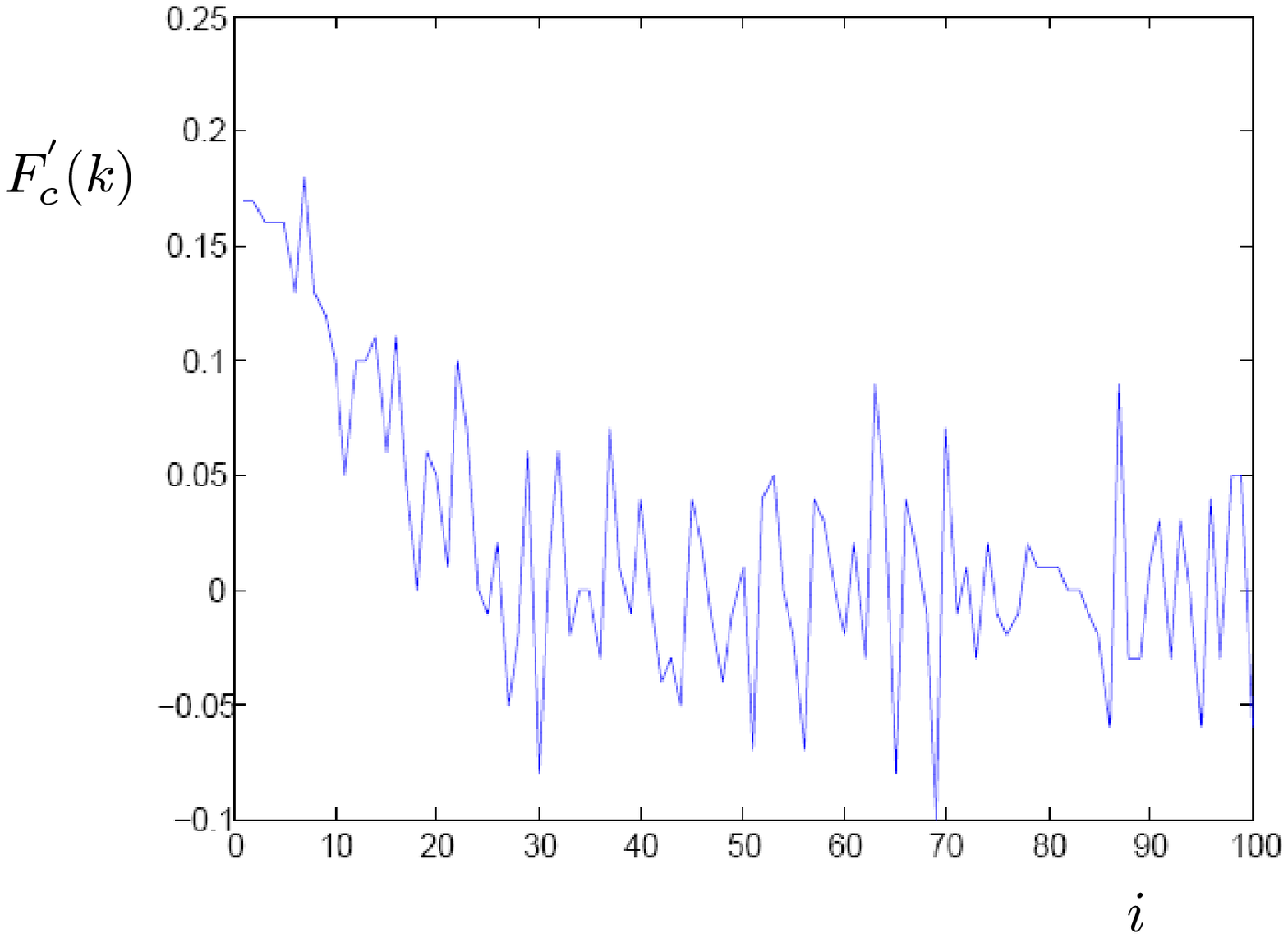}
\caption{$F_{c}^{'}(k)$ vs. iteration count, experiment 1.}
\label{fig:fig3}
\end{figure}

\begin{figure}[ht!]
\centering
\includegraphics[scale=0.40]{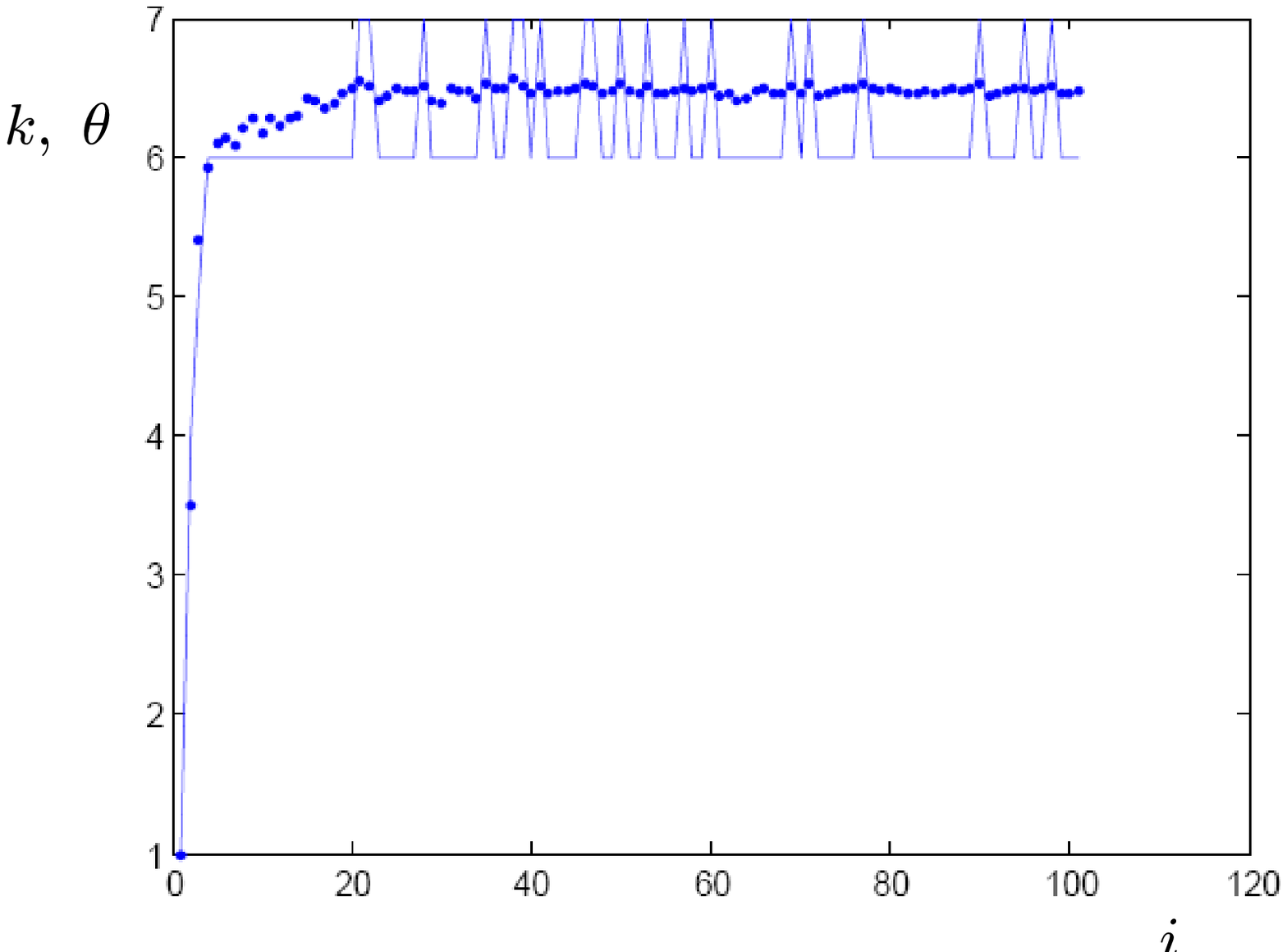}
\caption{$k$ and $\theta$ vs. iteration count, experiment 2.}
\label{fig:fig3}
\end{figure}

\begin{figure}[ht!]
\centering
\includegraphics[scale=0.39]{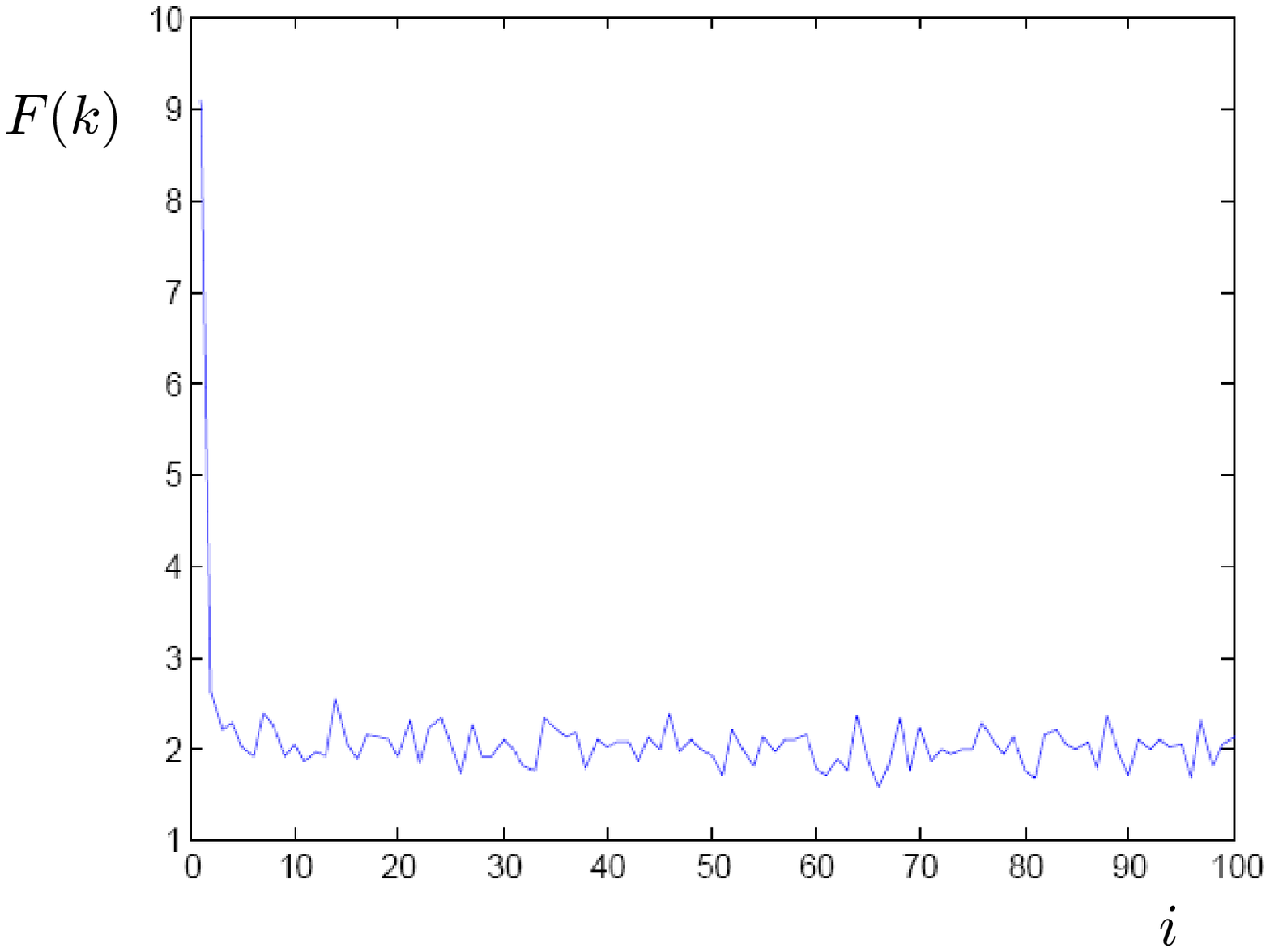}
\caption{ $F(k)$ vs.  iteration count, experiment 2.}
\label{fig:fig3}
\end{figure}

\begin{figure}[ht!]
\centering
\includegraphics[scale=0.39]{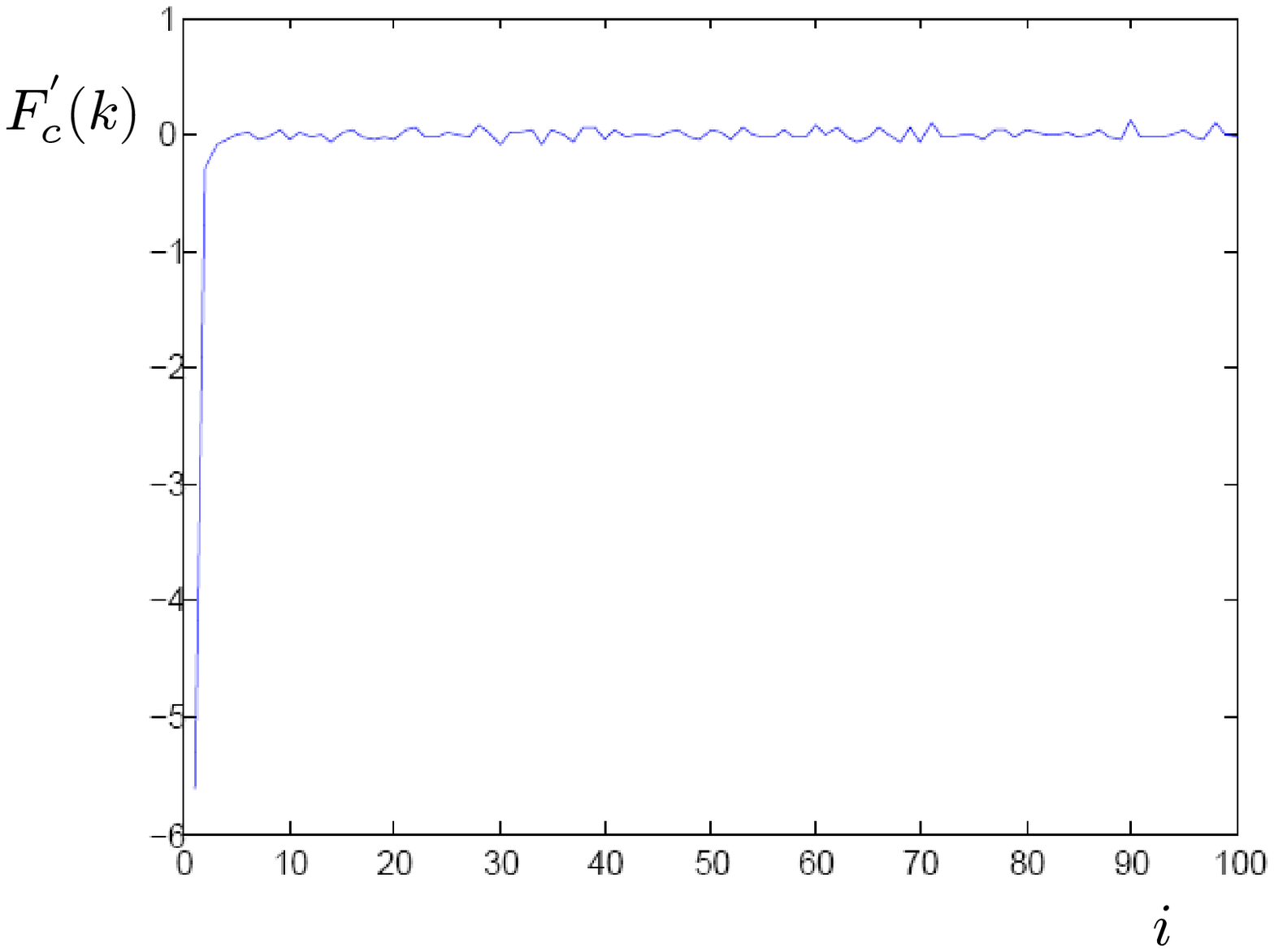}
\caption{$F_{c}^{'}(k)$ vs. iteration count, experiment 2.}
\label{fig:fig3}
\end{figure}

\end{document}